\documentclass[11pt,reqno]{amsart}

\usepackage[a4paper,margin=1in]{geometry}
\usepackage{graphicx}
\usepackage{mathtools}
\usepackage{subfig}
\usepackage{tabularx}
\usepackage{amsfonts}
\usepackage{amssymb}
\usepackage{algorithmic}
\usepackage{xcolor}
\usepackage{amsthm}
\usepackage{booktabs}
\usepackage{enumitem}

\newtheorem{theorem}{Theorem}[section] 
\newtheorem{lemma}[theorem]{Lemma}     
\newtheorem{corollary}[theorem]{Corollary}
\newtheorem{proposition}[theorem]{Proposition}

\newtheorem{definition}{Definition}[section]

\newtheorem{example}{Example}

\numberwithin{equation}{section}

\title{Tur\'an problem on the union of finite parallelepipeds}
\author{Xiao-Ye Fu}
\address{School of Mathematics and Statistics, Hubei Key Laboratory of Mathematical Sciences,
	Central China Normal University, Wuhan 430079, P.R. China}
\email{xiaoyefu@ccnu.edu.cn}

\author{Yue-Kun Li}
\address{School of Mathematics and Statistics, Hubei Key Laboratory of Mathematical Sciences,
	Central China Normal University, Wuhan 430079, P.R. China}
\email{m13937705619@163.com}

\author{Wei-Jie Wang}
\address{School of Mathematics and Statistics, Hubei Key Laboratory of Mathematical Sciences,
	Central China Normal University, Wuhan 430079, P.R. China}
\email{wwjmath@163.com}

\author{Xuan Wang\textsuperscript{*}}
\address{School of Mathematics and Statistics, Hubei Key Laboratory of Mathematical Sciences,
	Central China Normal University, Wuhan 430079, P.R. China}
\email{xuan\_wang@mails.ccnu.edu.cn}

\thanks{\textsuperscript{*}Corresponding author.}

\date{\today}


\begin{document}
	
	\maketitle
	
	\begin{abstract}
		This paper addresses the Tur\'an extremal problem on a class of origin-symmetric non-convex sets $\Omega \subset \mathbb{R}^d$.
		We establish a discrete approximation principle relating the continuous Tur\'an problem on $\Omega$ to a discrete Tur\'an problem on an associated finite discrete set. As an application, this enables the continuous Tur\'an problem on three equal length intervals
		$$\Omega_{\lambda}=(-1,1)\cup(\lambda-1,\lambda+1)\cup(-\lambda-1,-\lambda+1), \qquad \lambda\ge 0,$$
		to be exactly expressed as a limit of finite nonnegative trigonometric polynomial problems, which are equivalent to finite semidefinite programs. Using this framework, we compute the Tur\'an constant on three equal-length intervals for integer and rational parameters \(\lambda\). We further analyze the dependence of the Tur\'an constant on \(\lambda\).
	\end{abstract}

\section{Introduction}
	
	Let $\Omega \subset \mathbb{R}^d$ be open and symmetric about the origin. A continuous function \(f:\mathbb{R}^d \to \mathbb{C}\) is called \emph{Tur\'an admissible}  on $\Omega$ if it is positive definite, normalized so that \(f(0)=1\), and \(\operatorname{supp}f\subset\Omega\). The \emph{Tur\'an extremal problem} associated with $\Omega$ is to determine the supremum of $\int_\Omega f$ over all Tur\'an admissible functions on $\Omega$, i.e., 
	\begin{equation}\label{eq:turan-def}
		\mathcal{T}_{\mathbb{R}^d}(\Omega)=\sup\left\{\int_{\mathbb{R}}f(x)\,dx : 
		f\in C(\mathbb{R}^d),f\text{ is positive definite}, f(0)=1,\text{supp} f\subset\Omega \right\}.
	\end{equation}
	 This problem was proposed by Tur\'an and Stechkin \cite{Stechkin1972} and admits a natural formulation in every locally compact abelian (LCA) group \cite{Ivanov2006, IvanovIvanov2010, KR2006, Revesz2011}. Its origins go back to classical work of Carath\'eodory \cite{Caratheodory1911}, Fej\'er \cite{Fejer1916}, and Siegel \cite{Siegel1935} on related extremal questions for positive definite functions and trigonometric polynomials. In Euclidean space, the problem has been studied for a variety of specific domains, including the Euclidean ball \cite{Gabardo2024, Gorbachev2001, GorbachevManoshina2004, Siegel1935}, the regular hexagon and related polytopes \cite{ArestovBerdysheva2001, ArestovBerdysheva2002}, and convex bodies with spectral or translational tiling properties \cite{KLM2025, KR2003, KR2006, Revesz2011}. For the historical development of the Tur\'an problem and its extensions to various settings, we refer the reader to the survey of R\'ev\'esz \cite{Revesz2011}.
	
	\vspace{0.2cm}

	In one dimension, the exact Tur\'an constant is known for a single interval $\Omega=(-L,L)\subset\mathbb{R}$, where $L<\infty$, with $\mathcal{T}_{\mathbb{R}}((-L,L))=L$. It is abtained by taking $f=|\frac{\Omega}{2}|^{-1}\mathbf 1_\frac{\Omega}{2}*\mathbf 1_{-\frac{\Omega}{2}}$, where $|\Omega|$ denotes the Lebeague measure of a measurable set $\Omega$.	
	Disconnected or non-convex objects also fit naturally into the locally compact abelian group framework of Kolountzakis--R\'ev\'esz~\cite{KR2006}. One interesting example in~\cite{KR2006} is their one-dimensional result for an interval with two symmetric points removed: if $0<b<a\le 2b$ and
	$$
	\Omega=(-a,-b)\cup(-b,b)\cup(b,a),
	$$
	then the corresponding Tur\'an constant equals the Tur\'an constant for $(-b,b)$~\cite[Section~3.6]{KR2006}. 

	\vspace{0.2cm}
			
	The main goal in our work is to characterize the Tur\'an constant for a class of non-convex sets, for example, a union of three equal-length intervals on the real line. To solve this problem, we establish a discrete approximation principle, which converts the continuous Tur\'an problems associated with a origin-symmetric finite union of bounded open parallelepipeds into the limits of finite discrete Tur\'an problems. The explicit Tur\'an constant for a single symmetric interval is a classical result due to Stechkin \cite{Stechkin1972} and Gabardo\cite{Gabardo2024}, which we can also resolve with the general methods developed throughout this paper.

	\vspace{0.2cm}
		
	For a finite origin-symmetric subset $A \subset \mathbb{Z}^d$ with $0 \in A$, a complex sequence $(c_n)_{n \in \mathbb{Z}^d}$ is said to be Tur\'an admissible if the sequence is positive definite, satisfies $c_0=1$, and vanishes for all integers $n\notin A$. The discrete Tur\'an constant $\mathcal{T}_{\mathbb{Z}^d}(A)$ is then the supremum of the finite sum $\sum_{n\in A}c_n$ over all the discrete Tur\'an admissible sequences $(c_n)$, i.e.	
	\begin{equation}\label{eq:TZ-def}
		\mathcal{T}_{\mathbb{Z}^d}(A)=\sup\left\{\sum_{n\in A}c_n:(c_n)_{n\in \mathbb{Z}} ~\text{is positive definite},~ c_0 = 1, ~c_n = 0 ~\text{for}~ n \notin A \right\}.
	\end{equation}

	\vspace{0.2cm}
	
    Denote $\operatorname{int}\Omega$ the interior of $\Omega$. Let $\Omega \subset \mathbb{R}^d$ be an origin-symmetric finite union of similar open parallelepipeds with parallel sides, and suppose $0 \in \operatorname{int} \Omega$, that is 
    \begin{equation}\label{eq:nonconvex-Omega}
	\Omega = \bigcup_{i=-m}^m \left(s_i M (-1,1)^d+ t_i\right)
    \end{equation}
	where $s_i > 0$ and $s_i = s_{-i}$, $M$ is a $d \times d$ invertible matrix and $t_i \in \mathbb{R}^d$ satisfying $t_0 = 0$ and $t_i=- t_{-i}$ for $i=1,\cdots,m$. For $h>0$, we denote 
	$$
	A_{M,h}^- = \left\{n \in \mathbb{Z}^d : hM\left(n + (-1,1)^d \right) \subset \Omega \right\}.
	$$
	We now state the discrete approximation principle that reveals the connection between  continuous Tur\'an problem and finite discrete Tur\'an problem.
	\begin{theorem}\label{th:approximate}
		Let $\Omega \subset \mathbb{R}^d$ be as in \eqref{eq:nonconvex-Omega}. Then
		\begin{equation}\label{eq:finite-union-limit}
			\mathcal{T}_{\mathbb{R}^d} (\Omega) = |\det M| \lim_{ N \to\infty} \frac{1}{N^d}  \mathcal{T}_{\mathbb{Z}^d}(A_{M,\frac{1}{N}}^-).
		\end{equation}
	\end{theorem}
	
	Using this result, we analyze the Tur\'an problem on the three-interval set
	\begin{equation}\label{eq:Omega-lambda}
	\Omega_{\lambda}=(-1,1)\cup(\lambda-1,\lambda+1)\cup(-\lambda-1,-\lambda+1), \qquad \lambda\ge 0.
	\end{equation}
	and provide two concrete examples, corresponding to integer values of $\lambda$ and to the rational value $\lambda=5/2$. In particular, we observe some interesting behavior of the Tur\'an constant $\mathcal{T}_{\mathbb{R}}(\Omega_\lambda)$ as a function of $\lambda$, including the discontinuity at the critical threshold $\lambda=2$ and the limit as $\lambda \rightarrow \infty$.
	
	
	\vspace{0.2cm}
			
	The paper is organized as follows. In Section 2, we introduce the notations of positive definite functions and positive definite sequences, and establish their interrelations needed for later arguments. In Section 3, we establish the fundamental connection between continuous Tur\'an problems on $\mathbb{R}^d$ and discrete Tur\'an problems on $\mathbb{Z}^d$. 
	In Section 4, we apply the approximation theory to the Tur\'an problem on the three-interval set and analyze the qualitative behavior of $\mathcal{T}_{\mathbb{R}}(\Omega_\lambda)$ across the parameter $\lambda$. In section 5, we conclude the paper with a discussion of open problems and prospective generalizations to multiple separated symmetric interval blocks.

\section{Preliminaries}
In this section, we introduce positive definite functions and positive definite sequences, and establish their connection for subsequent use. For every $f \in L^1(\mathbb{R}^d)$, we define its Fourier transform by the formula 
	$$
	\widehat{f} (\xi) =\int_{\mathbb{R}^d} f(x) e^{-2 \pi i \langle\xi,  x\rangle } dx.
	$$
	A continuous function $f:\mathbb{R}^d \to \mathbb{C} $ is positive definite if
	$$
	\sum_{j,k=1}^m z_j\overline{z_k}f(x_j-x_k)\ge0
	$$
	for all finite choices of $x_1,\cdots,x_m\in \mathbb{R}^d $ and $z_1,\cdots,z_m\in\mathbb{C} $. Bochner’s theorem states that the positive definiteness of an integrable function $f$ is equivalent to the positivity of its Fourier transform, i.e. $\widehat{f} \geq 0$. For more details, one can refer to~\cite{Bochner1932, Rudin1962, Sasvari1994}. The following are several simple properties of positive definite functions.

	\begin{lemma}[\cite{Sasvari2013}]\label{lem:elementary-pd} 
		Let $f$ be a positive definite function on $\mathbb{R}^d$. Then
		\begin{enumerate}[label=\textup{(\roman*)},leftmargin=2.2em]
			\item $f$ is Hermitian, i.e. $f(-x)=\overline{f(x)}$.
			\item $|f(x)|\le f(0)$.
			\item The functions $\overline{f},  \ \operatorname{Re}f, \ \left|f\right|^2$ are positive definite. 
			\item If $g$ is another positive definite function on $\mathbb{R}^d$, then $p_1 f+p_2 g$ is positive definite for all $p_1, p_2 \geq 0$.
		\end{enumerate}
	\end{lemma}
	
	As a consequence of Lemma \ref{lem:elementary-pd} and the definition of $\mathcal{T}_{\mathbb{R}^d}(\Omega)$, one may restrict the functions in \eqref{eq:turan-def} to be real-valued, even, continuous positive definite functions, which we will only consider in the following content.

	\vspace{0.2cm}
		
	For functions supported on lattice sets, that is, for sequences, we could similarly define the positive definite sequences.

	\begin{definition}
		Let $c=(c_m)_{m \in A\mathbb{Z}^d}$ be a complex-valued sequence, where $A$ is a $d \times d$ invertible matrix. We say that $c$ is positive definite if for every finite set $\{m_1,\cdots,m_r\} \subset A\mathbb{Z}^d$ and every $z_1,\cdots,z_r\in\mathbb{C}$, we have
		$$\sum_{j,k=1}^r z_j \overline{z_k} c_{m_j-m_k}\geq 0.$$
	\end{definition}
	
	\begin{lemma}\label{lem:finite-herglotz}
		Let $c=(c_m)_{m \in A\mathbb{Z}^d}$ be a finitely supported sequence, which means the set $\{ m \in A\mathbb{Z}^d : c_m \neq 0\}$ is a finite set. Then $c$ is positive definite if and only if
		\begin{equation}\label{eq:trigpoly-positive}
			P_c(t)=\sum_{m\in A\mathbb{Z}^d}c_m e^{2 \pi i \langle m, t\rangle}\ \ge\ 0\qquad\text{for every }t \in \mathbb{R}^d.
		\end{equation}
	\end{lemma}
	
	This lemma constitutes the finite‑support version of the Herglotz--Bochner theorem (see Herglotz~\cite{Herglotz1911}, Rudin~\cite[Section~1.4]{Rudin1962}, or Katznelson~\cite[Chapter~I]{Katznelson2004}). Furthermore, positive definite sequences share numerous properties with positive definite functions, including $c_{-m} = \overline{c_m}$ and $|c_m| \leq c_0$.

	\vspace{0.2cm}	
    The next lemma is a key observation of the present work: every positive definite function generates a family of positive definite sequences via sampling.
	
	\begin{lemma}\label{lem:conti-discete}
		Let $f$ be a positive definite function on $\mathbb{R}^d$. Then for every lattice set $M \mathbb{Z}^d$, where $M$ is a $d \times d$ invertible matrix, the sampled sequence
		$\{f(m)\}_{m \in M \mathbb{Z}^d}$
		is positive definite.
	\end{lemma}
	\begin{proof}
		For arbitrary $r$ points $m_1,\cdots,m_r \in M \mathbb{Z}^d$ and complex numbers $z_1,\cdots,z_r$, we have 
		$$
	\sum_{j,k=1}^r z_j\overline{z_k}f(m_j-m_k) \ge 0
		$$
		since $f$ is a positive definite function. 
	\end{proof}
	

	
	\begin{lemma}\label{lem:periodization}
		Let $f\in C_c(\mathbb{R}^d)$ be positive definite and let $M$ be a $d \times d$ invertible matrix. Define 
		$$P(x)=\sum_{m \in M\mathbb{Z}^d} f(x+m).$$ Then $P$ is a continuous positive definite function, and
		\begin{equation}\label{eq:periodization-ineq}
			\int_{\mathbb{R}^d} f(x) \,dx \leq  |\det M| \sum_{m \in M\mathbb{Z}^d} f(m) = |\det M|P(0) .
		\end{equation}
	\end{lemma}
	\begin{proof}
		Since $f$ is continuous, bounded and compactly supported, the series defining $P$ is locally finite. Thus $P$ is continuous, bounded and periodic with period lattice \(\Lambda=M\mathbb Z^d\). 
		Denote the dual lattice 
		$$\Lambda^* = \{ \xi \in \mathbb{R}^d \mid \langle \xi, k \rangle \in \mathbb{Z}, \quad \forall k \in M\mathbb{Z}^d \} =  M^{-T} \mathbb{Z}^d.$$
		Let  $\mathcal{D} = \mathbb{R}^d / M\mathbb{Z}^d$ be a fundamental domain for \(\Lambda\).
		Then for each $\xi \in \Lambda^*$, the Fourier coefficients of $P$ on the torus \(\mathbb R^d/\Lambda\) satisfies
		\begin{align*}
			\widehat{P} (\xi) &= \frac{1}{|\det M|} \int_{\mathcal{D}} P(x) e^{-2\pi i \langle x, \xi \rangle}  dx\\
			&=\frac{1}{|\det M|} \int_{\mathcal{D}} \sum_{k \in \Lambda} f(x+k) e^{-2\pi i \langle x, \xi \rangle}  dx\\
			&=\frac{1}{|\det M|} \sum_{k \in \Lambda} \int_{\mathcal{D}} f(x+k) e^{-2\pi i \langle x, \xi \rangle}  dx\\
			&=\frac{1}{|\det M|} \int_{\mathbb{R}^d} f(y) e^{-2\pi i \langle y, \xi \rangle}  dy\\
			&=\frac{1}{|\det M|} \widehat{f}(\xi)
		\end{align*}
	Since f is continuous and positive definite, Bochner’s theorem gives \(\widehat f(\xi)\ge0\) for all \(\xi\in\mathbb R^d\). Therefore \(\widehat P(\xi)\ge0\) for all \(\xi\in\Lambda^*\). It follows that the function $P$ is positive definite. This implies that 
		$$
		\frac{1}{|\det M|} \int_{\mathbb{R}^d} f(x) dx = \frac{1}{|\det M|} \int_{\mathcal{D}} P(x) dx \leq \frac{1}{|\det M|} \int_{\mathcal{D}} P(0) dx = P(0).
		$$
	\end{proof}

	Now we use lemma \ref{lem:periodization} to prove the Tur\'an extremal problem on a single parallelepiped. 
	For $h>0$, we define 
	\begin{equation}\label{eq:triangle-properties}
		\phi_h=\frac{1}{h^d} \chi_{(-h/2,h/2)^d} \ast \chi_{(-h/2,h/2)^d}.
			\end{equation}
 We can directly check that
	$$
\operatorname{supp} \phi_h= [-h,h]^d, \quad	\phi_h(0)=1, \quad\int_{\mathbb{R}^d}\phi_h(x)\,dx=h^d.
	$$
	Indeed, $\phi_h$ is an autocorrelation function and thus \(\phi_h\) is positive definite. 
	
	\begin{proposition}\label{prop:connected}
		Let $M$ be a $d \times d$ invertible matrix and $\Omega = M (-1,1)^d \subset \mathbb{R}^d$. Then 
		$$\mathcal{T}_{\mathbb{R}^d}(\Omega)= |\det M|.$$
	\end{proposition}
	
	\begin{proof}
		Define the function 
		$$
		g(x) = \chi_{M(-1/2, 1/2)^d}(x).
		$$	
		and set 
		$$
		\phi(x) = \frac{1}{|\det M|} (g * \widetilde{g})(x)  
		$$
		Then $\phi(x)$ is continuous and positive definite, $\phi (0) = 1 $ and $\operatorname{supp} \phi = \overline{ \Omega } $ by \eqref{eq:triangle-properties}. It follows that $\phi(x)$ is admissible for the Tur\'an extremal problem on $\Omega$. Thus 
		$$
		\mathcal{T}_{\mathbb{R}^d}(\Omega) \geq  \int_{\mathbb{R}^d}  \phi(x) \,dx = |\det M|.
		$$
		
		On the other hand, let $f$ be an admissible function for Tur\'an extremal problem on $\Omega$. Applying Lemma~\ref{lem:periodization}, we have
		$$
		\int_{\mathbb{R}^d} f(x) \,dx \leq |\det M| P(0) = |\det M|.
		$$
		Indeed, since $f$ is continuous and $\operatorname{supp} f \subset \Omega$, we have $f=0$ on the boundary $\partial \Omega$. Then $P(0)=\sum_{m \in M\mathbb{Z}^d} f(m) = f(0) = 1$. Hence 
		$$
		\mathcal{T}_{\mathbb{R}^d}(\Omega) \leq |\det M| .
		$$
	\end{proof}
	
	\section{Connection between the continuous and discrete Tur\'an problem}
	
    In this section, we establish the connection between the Turán extremal problem over a finite union of parallelepipeds and its discrete counterpart on a finite lattice set. Let $\Omega \subset \mathbb{R}^d$ be an origin-symmetric open set. Assume that $0\in \operatorname{int}\Omega$. Let $M$ be a $d \times d$ invertible matrix. For $h>0$, we denote 
    \begin{equation*}\label{eq:Aminus}
    A_{M,h}^- (\Omega) = \left\{n \in \mathbb{Z}^d : hM\left(n + (-1,1)^d \right) \subset \Omega \right\}, 
    \quad
    A_{M,h}^+ (\Omega) = \left\{n \in \mathbb{Z}^d : hMn  \in  \Omega \right\}.
    \end{equation*}
    Note that $A^-_{M,h}$ and $ A^+_{M,h}$ are finite sets whenever $\Omega$ is a bounded set.

	\begin{theorem}\label{thm:sandwich}
		Let $\Omega \subset \mathbb{R}^d$ be a origin-symmetric finite union of similar open parallelepipeds with parallel sides, and suppose $0 \in \operatorname{int} \Omega$, that is 
	$$
	\Omega = \bigcup_{i=-m}^m \left(s_i M (-1,1)^d+ t_i\right)
	$$
	where $s_i > 0$ and $s_i = s_{-i}$, $M$ is a $d \times d$ invertible matrix and $t_i \in \mathbb{R}^d$ satisfying $t_0 = 0$ and $t_i=- t_{-i}$ for $i=1,\cdots,m$. If $h > 0$ is small enough such that $M(-h,h)^d +t_i \subset \Omega $ for each $i= -m, \cdots, m$, then
	\begin{equation}\label{eq:sandwich}
		h^d |\det M| \mathcal{T}_{\mathbb{Z}^d}(A_{M,h}^-) \le 
		\mathcal{T}_{\mathbb{R}^d} (\Omega) \le 
		h^d |\det M|\mathcal{T}_{\mathbb{Z}^d}(A_{M,h}^+).
	\end{equation}
	\end{theorem}
	
	\begin{proof}
		We prove the lower and upper bound separately.
		
		\smallskip
		\noindent\textbf{Step 1: upper bound.} 
		Let $f$ be admissible for $\Omega$. Without lost of generality, we may assume $f$ is real and even. 
		Lemma~\ref{lem:periodization} gives
		$$
		\int_{\mathbb{R}^d} f(x) dx \leq h^d |\det M| \sum_{m \in \mathbb{Z}^d} f(hMm)
		$$
		Set $c_m=f(hMm)$. Then the sequence $c=(c_m)_{m \in \mathbb{Z}^d}$ satisfies that $c_0 = 1$ and $c_m = 0$ for $m \notin A_{M,h}^+$. By Lemma \ref{lem:conti-discete}, we know that the sequence $c$ is positive definite. Then $c$ is an admissible sequence and we have 
		$$
		\int_{\mathbb{R}^d} f(x) dx\leq  h^d |\det M| \mathcal{T}_{\mathbb{Z}^d} (A_{M,h}^+).
		$$
	Taking the supremum over all admissible functions $f$, we obtain	
		$$
		\mathcal{T}_{\mathbb{R}^d} (\Omega) \leq h^d |\det M| \mathcal{T}_{\mathbb{Z}^d} (A_{M,h}^+).
		$$

		\smallskip
		\noindent\textbf{Step 2: lower bound.}
		Let $c=(c_n)_{n \in \mathbb{Z}^d}$ be positive definite sequence with $c_0=1$ and $c_n=0$ for $n \notin A_{M,h}^-$. For small enough $\epsilon > 0$, we define the lifted function 
		\begin{equation}\label{eq:lifting-function}
			f(x)=\sum_{n\in A_{M,h}^-}c_n\phi_{\epsilon} (x-hMn).
		\end{equation}
		where $\phi_{\epsilon}(x) = \frac{1}{|\det (1-\epsilon)hM|} (g * \widetilde{g})(x) $ and $g_{\epsilon}(x)= \chi_{(1-\epsilon)hM(-1/2, 1/2)^d}(x)$.
		By the definition of $A_{M,h}^-$, the function $f$ is integrable, $f(0) = c_0 \phi_{\epsilon} (0) = 1$ and supported on $\Omega$. The Fourier transform of $f$ is
		$$
		\widehat{f}(\xi) = \widehat{\phi_{\epsilon}}(\xi) \sum_{n\in A_{M,h}^-} c_n  e^{-2\pi i h \langle M^T \xi, n\rangle }.
		$$
		Lemma \ref{lem:finite-herglotz} implies that $\sum_{n\in A_{M,h}^-} c_n  e^{-2\pi i h  \langle M^T \xi, n\rangle} \geq 0$. Since $\phi_{\epsilon} (\xi)$ is a positive definite function, we have $\widehat{f} (\xi) \geq 0$. Then the function $f(x)$ is Tur\'an admissible and 
		$$
		\int_{\mathbb{R}^d} f(x) \,dx=\sum_{n\in A_{M,h}^-} c_n\int_{\mathbb{R}^d} \phi_{\epsilon}(x-hMn) dx = (1-\epsilon)^d h^d |\det M| \sum_{n\in A_{M,h}^-}c_n. 
		$$
		Thus,
		$$
		(1-\epsilon)^d h^d |\det M|\sum_{n\in A_{M,h}^-}c_n \leq \mathcal{T}_{\mathbb{R}^d}(\Omega)
		$$
		for all admissible sequence $(c_n)_{n \in \mathbb{Z}^d}$. Taking $\epsilon \rightarrow 0$, we obtain $h^d |\det M| \mathcal{T}_{\mathbb{Z}^d}(A_{M,h}^-) \le \mathcal{T}_{\mathbb{R}^d} (\Omega)$.
	\end{proof}
	
	We note that the inequality (\ref{eq:sandwich}) becomes an identity if $A_{M,h}^-=A_{M,h}^+$. However, in general, the set $A^-_{M,h}$ is contained in, but not equal to, the set $A^+_{M,h}$.
	
	\begin{corollary}\label{cor:rational-endpoints}
	Let $\Omega \subset \mathbb{R}^d$ be a origin-symmetric finite union of similar open parallelepipeds with parallel sides, and suppose $0 \in \operatorname{int} \Omega$, that is 
		$$
		\Omega = \bigcup_{i=-m}^m \left(s_i M (-1,1)^d+ t_i\right),
		$$
		where $s_i > 0$ and $s_i = s_{-i}$, $M$ is a $d \times d$ invertible matrix and $t_i \in \mathbb{R}^d$ satisfying $t_0 = 0$ and $t_i=- t_{-i}$ for $i=1,\cdots,m$. 
		Assume that all $s_i$ belong to $q^{-1}\mathbb{Z}$ and $t_i$ belong to $q^{-1}M\mathbb{Z}^d$ for some integer $q \ge 1$. Then
		\begin{equation}\label{eq:rational-endpoints}
			\mathcal{T}_{\mathbb{R}^d}(\Omega)=\frac{|\det M|}{q^d} \, \mathcal{T}_{\mathbb{Z}^d}\left(\left\{n\in\mathbb{Z}^d: \frac{1}{q} Mn \in \Omega \right\}\right).
		\end{equation}
	\end{corollary}
	
	\begin{proof}
		Note that 
		\begin{equation*}
			A_{M,1/q}^- (\Omega) = \left\{n \in \mathbb{Z}^d :  q^{-1}M\left(n + (-1,1)^d \right) \subset \Omega \right\}, 
			\quad
			A_{M,1/q}^+ (\Omega) = \left\{n \in \mathbb{Z}^d :  q^{-1}Mn  \in \Omega \right\}.
		\end{equation*}
		It's obvious that $A_{M,1/q}^- (\Omega) \subset A_{M,1/q}^+ (\Omega)$. 
		
		Conversely, for each $n \in A_{M,1/q}^+ (\Omega)$, that is $q^{-1}Mn  \in \Omega$, there exist some $i \in \{-m, \cdots,0,\cdots,m \}$ such that 
		$$
		q^{-1}Mn  \in \left( s_i M (-1,1)^d+ t_i\right).
		$$
		By the assumption, we denote $s_i = q^{-1} k_i, t_i =q^{-1}Ml_i$ with $k_i \in \mathbb{Z}, l_i \in \mathbb{Z}^d$. Then we have 
		$$
		n  \in \left( (-k_i,k_i)^d+ l_i\right).
		$$
		It follows that $n + (-1,1)^d \subseteq \left( (-k_i,k_i)^d+ l_i\right)$ which means $n \in A_{M,1/q}^- (\Omega)$.	
		Hence $$A_q^+(\Omega)=A_q^-(\Omega).$$ Applying Theorem~\ref{thm:sandwich} with $h=1/q$, we obtain the identity (\ref{eq:rational-endpoints}).
	\end{proof}

	\subsection{The discrete approximation principle}
	
	In contrast to the above discussion, we now show that the Tur\'an constant on finite unions of parallelepipeds could be approximated via a discrete extremal problem on a finite discrete set.
	Let $\Omega \subset \mathbb{R}^d$ be a symmetric finite union of open parallelepipeds with $0 \in \operatorname{int}\Omega$. For all sufficiently large $N$, set
	\begin{equation}\label{eq:AN-Omega}
		A_N (\Omega) = \left\{ n \in \mathbb{Z}^d: \frac{1}{N} M\left(n + (-1,1)^d \right) \subset \Omega \right\}.
	\end{equation}
	Indeed $A_N(\Omega)=A^-_{M,1/N}$. We denote $A^+_{M,1/N}$ as $A_N^+(\Omega)$.
	
	\begin{theorem}\label{thm:finite-union-limit}
		Let $\Omega \subset \mathbb{R}^d$ be an origin-symmetric finite union of bounded open parallelepipeds with parallel sides, and suppose $0 \in \operatorname{int} \Omega$. Then
		\begin{equation}\label{eq:finite-union-limit}
			\mathcal{T}_{\mathbb{R}^d} (\Omega) = |\det M| \lim_{ N \to\infty} \frac{1}{N^d}  \mathcal{T}_{\mathbb{Z}^d}(A_N(\Omega)).
		\end{equation}
	\end{theorem}
	
	\begin{proof}
		Taking $h = \frac{1}{N}$ in Theorem \ref{thm:sandwich}, we have $\frac{1}{N^d}|\det A| \mathcal{T}_{\mathbb{Z}^d} \left(A_N(\Omega)\right) \le \mathcal{T}_{\mathbb{R}^d} (\Omega)$ and then
		\begin{equation}\label{eq:finite-union-limsup}
			\limsup_{N\to+\infty} \frac{1}{N^d}|\det A| \mathcal{T}_{\mathbb{Z}^d}\left(A_N(\Omega)\right) \le \mathcal{T}_{\mathbb{R}^d} (\Omega).
		\end{equation}
		Define 
		$
		B_N (\Omega) =A_N^+(\Omega) \setminus A_N(\Omega).
		$
		Then $\# B_N = O(N^{d-1})$. 
	For every Tur\'an admissible function $f$ on $\Omega$, $\operatorname{supp}(f) \subset \Omega$ implies that  there exists $\delta > 0$ such that 
		$$
		\operatorname{dist} (\partial \Omega, \operatorname{supp} (f)) \geq \delta.
		$$
		So we can pick  $N$ large enough such that $2|\operatorname{det} M|^{1/d}/N < \delta$ and for all $x \in B(\partial \Omega, 2|\operatorname{det} M|^{1/d}/N)$, we have $f(x) = 0<1/N^d$. This together with  $\# B_N = O(N^{d-1})$ gives that
		\begin{equation}\label{eq:edge-estimation}
			\eta_N=\sum_{n\in B_N} \left| f \left(\frac{1}{N} M n \right) \right| \leq \frac{C}{N},
		\end{equation}	
		for some constant $C \in \mathbb{R}$. Thus, $ \eta_N\longrightarrow 0 ~\text{as}~ N \to +\infty$. Define a sequence $\{c_n^{(N)}\}_{n \in \mathbb{Z}^d}$ as 
		$$
		c_0^N = 1+ \eta_N,\quad c_n^{(N)}=f(\frac{1}{N} M n)\ (n\in A_N(\Omega),\ n\ne0),\quad
		c_n^{(N)}=0\ (n\notin A_N(\Omega)).
		$$
		Indeed, $\{c_n^{(N)}\}_{n \in \mathbb{Z}^d}$ is a positive definite sequence. More precisely, for $t \in \mathbb{R}^d$, we denote 
		$$
		p_N(t)=\sum_{n\in A_N^+(\Omega)} f(\frac{1}{N} M n)e^{2 \pi i \frac{1}{N} \langle Mn, t\rangle},\qquad
		q_N(t)=\sum_{n\in A_N(\Omega)} f(\frac{1}{N} M n)e^{2 \pi i \frac{1}{N} \langle Mn, t\rangle}.
		$$
		Since $f$ is real and even, and the index sets $A_N^+(\Omega)$ ans $A_N(\Omega)$ are symmetric, both trigonometric polynomials $p_N(t)$ and $q_N(t)$ are real-valued. Lemma~\ref{lem:finite-herglotz} gives $p_N(t)\ge0$. Moreover, 
		$$
		\left|\sum_{n\in B_N(\Omega)} f(\frac{1}{N} M n)e^{2 \pi i \frac{1}{N} \langle Mn, t\rangle} \right|
		\le\sum_{n\in B_N(\Omega)}|f(\frac{1}{N} M n)|=\eta_N.
		$$
		Hence,
		$$
		q_N(t)=p_N(t)-\sum_{n\in B_N(\Omega)}f(\frac{1}{N} M n)e^{2 \pi i \frac{1}{N} \langle Mn, t\rangle}\ge -\eta_N,
		$$
		that is $q_N(t) + \eta_N \geq 0$, which implies the positive-definiteness of the sequence $\{c_n^{(N)}\}_{n \in \mathbb{Z}^d}$. 
		Therefore, the sequence $\{\frac{1}{1+ \eta_N}c_n^{(N)}\}_{n \in \mathbb{Z}^d}$ is admissible for $\mathcal{T}_{\mathbb{Z}^d}(A_N(\Omega))$ and 
		\begin{equation}\label{eq:general-TZ-from-f}
			\frac{|\det M|}{N^d} \mathcal{T}_{\mathbb{Z}^d}(A_N(\Omega))
			\ge
			\frac{|\det M|}{(1+\eta_N)N^d} \left( \sum_{n\in A_N(\Omega)} f(\frac{1}{N} M n) + \eta_N \right).
		\end{equation}	
		The Riemann sums of function $f \in C_c (\mathbb{R}^d)$ satisfy
		$$
		\frac{|\det M|}{N^d}\sum_{n\in\mathbb{Z}^d}f(\frac{1}{N} M n) \longrightarrow \int_{\mathbb{R}^d} f(x)dx,~\text{as}~ N \longrightarrow +\infty .
		$$
		Recall that $A_N^+(\Omega) = A_N(\Omega) \cup B_N(\Omega)$ with disjoint union and $f(n/N)=0$ for $n \notin A_N^+(\Omega) $. Then
		$$
		\frac{1}{N^d} \sum_{n\in\mathbb{Z}^d} f(\frac{1}{N} M n)-
		\frac{1}{N^d} \sum_{n\in A_N(\Omega)}f(\frac{1}{N} M n)
		=\frac{1}{N^d} \sum_{n\in B_N(\Omega) }f(\frac{1}{N} M n).
		$$
		Consequently
		$$
		\left|\frac{1}{N^d} \sum_{n\in\mathbb{Z}^d} f(\frac{1}{N} M n)-
		\frac{1}{N^d} \sum_{n\in A_N(\Omega)}f(\frac{1}{N} M n) \right|
		\le \frac{\eta_N}{N^d}.
		$$
		Letting $ N\to +\infty$ in \eqref{eq:general-TZ-from-f}, we obtain that 
		$$
		\liminf_{N\to\infty}\frac{|\det M|}{N^d} \mathcal{T}_{\mathbb{Z}^d}(A_N(\Omega) )
		\ge 
		\int_{\mathbb{R}^d} f(x)\,dx.
		$$
		Taking the supremum over all admissible $f$ proves
		$$
		\liminf_{N\to\infty} \frac{|\det M|}{N^d} \mathcal{T}_{\mathbb{Z}^d}(A_N(\Omega) ) 
		\ge 
		\mathcal{T}_{\mathbb{R}^d} (\Omega).
		$$
		Together with \eqref{eq:finite-union-limsup}, this proves the theorem.
		
	\end{proof}
	
\section{Applications to three intervals}\label{algorithm}
	Now we apply the results of the Tur\'an extremal problem on finite parallelepipeds in Subsection 2.2 to a special case, the Tur\'an extremal problem on a union of three equal-length intervals on the real line. Without loss of generality, we work with the normalized one-parameter family of origin-symmetric open sets
	\begin{equation}\label{eq:Omega-lambda}
		\Omega_{\lambda}=(-1,1)\cup(\lambda-1,\lambda+1)\cup(-\lambda-1,-\lambda+1), \qquad \lambda\ge 0.
	\end{equation}
	We first apply Corollary~\ref{cor:rational-endpoints} to the case $\lambda=p/q$ with $\gcd(p,q)=1$.
	
	\begin{theorem}\label{thm:rational}
		Let $\lambda=p/q>2$ with $p,q\in N$ and $\gcd(p,q)=1$. Let
		\begin{equation}\label{eq:Apq}
			A_{p,q}=\{n\in \mathbb{Z} :|n|<q\}\cup\{n\in\mathbb{Z}:|n-p|<q\}\cup\{n\in \mathbb{Z}:|n+p|<q\}.
		\end{equation}
		Then
		\begin{equation}\label{eq:rational-formula}
			\mathcal{T}_{\mathbb{R}}(\Omega_{p/q})=\frac{1}{q}\,T_{\mathbb{Z}}(A_{p,q}).
		\end{equation}
	\end{theorem}
	
	\begin{proof}
		Let $\lambda=p/q>2$ with $\gcd(p,q)=1$. Then
		$$
		\Omega_{p/q}= \left(-1,1\right)\cup \left(\frac{p}{q}-1,\frac{p}{q}+1 \right)\cup \left(-\frac{p}{q}-1,-\frac{p}{q}+1 \right).
		$$
		All the boundary points of $\Omega_{p/q}$ lie in the lattice set $q^{-1}\mathbb{Z}$, so Corollary~\ref{cor:rational-endpoints} gives
		$$
		\mathcal{T}(\Omega_{p/q})=\frac1q\,\mathcal{T}_\mathbb{Z}\left(\left\{n\in\mathbb{Z}:\frac nq\in\operatorname{int}\Omega_{p/q}\right\}\right).
		$$	
		Obviously, we have the following three equivalence relations, 
		$$
		\frac {n}{q}\in(-1,1)\Longleftrightarrow |n|<q,
		$$
		$$
		\frac {n}{q}\in \left(\frac{p}{q}-1,\frac{p}{q}+1 \right) \Longleftrightarrow |n-p|<q,
		$$
		and
		$$
		\frac {n}{q}\in \left(-\frac{p}{q}-1,-\frac{p}{q}+1 \right) \Longleftrightarrow |n+p|<q.
		$$
		Thus
		$$
		\left\{n\in\mathbb{Z}:\frac {n}{q} \in \operatorname{int}\Omega_{p/q}\right\} = A_{p,q}.
		$$
	By Corollary~\ref{cor:rational-endpoints}, the statement in Theorem~\ref{thm:rational} holds.
	\end{proof}

	Now we specialize Theorem~\ref{thm:finite-union-limit} to $\Omega_\lambda$ with arbitrary $\lambda$.
	
	\begin{theorem}\label{thm:limit}
		Let $\lambda>2$ be arbitrary. For each integer $N \ge 1$, define
		\begin{equation}\label{eq:AN}
			A_N(\lambda)=\{n\in \mathbb{Z}:|n|\le N-1\}\cup\{n\in\mathbb{Z} :|n-N\lambda|\le N-1\}\cup\{n\in\mathbb{Z} :|n+N\lambda|\le N-1\}.
		\end{equation}
		Then
		\begin{equation}\label{eq:limit-formula}
			\mathcal{T}_{\mathbb{R}} (\Omega_\lambda)=\lim_{N\to\infty} \frac{1}{N} \mathcal{T}_{\mathbb{Z}} (A_N(\lambda)).
		\end{equation}
	\end{theorem}
	
	\begin{proof}
		Let $\lambda>2$. For the interval $(-1,1)$, the condition
		$$
		\frac nN+\left(-\frac1N,\frac1N\right)\subset(-1,1)
		$$
		is equivalent to $|n|\le N-1$.  For the interval $(\lambda-1,\lambda+1)$, the condition 
		$$	
		\frac nN+\left(-\frac1N,\frac1N\right)\subset(\lambda-1, \lambda+1)
		$$
	 is equivalent to
		$|n-N\lambda|\le N-1$. Similarly for the interval $(-\lambda-1,-\lambda+1)$, we have $|n+N\lambda|\le N-1$.  Thus
		$$
		A_N(\Omega_\lambda)=
		\{n:|n|\le N-1\}\cup\{n:|n-N\lambda|\le N-1\}\cup\{n:|n+N\lambda|\le N-1\}.
		$$
		Substituting this identity into \eqref{eq:finite-union-limit} gives \eqref{eq:limit-formula}.
	\end{proof}
	
	\subsection{Examples of discrete Tur\'an problems}\label{subsec:examples}
 We now illustrate this reduction with two concrete cases. For integer parameters, the associated discrete set contains only three points and the Tur\'an constant can be computed exactly. We then consider the non-integer rational value $\lambda=5/2$ and construct an explicit nonnegative trigonometric polynomial showing that the Tur\'an constant is strictly larger than $2$. A semidefinite-program representation, useful for numerical computation of the finite discrete problems, is provided separately in Appendix~A.
	
	
	\begin{example}\label{cor:integer}
		If $N \ge 3$ is an integer, then $\mathcal{T}_{\mathbb{R}}(\Omega_N)=2$.
	\end{example}
	
	\begin{proof}
		By Theorem~\ref{thm:rational}, We just need to calculate  $\mathcal{T}_\mathbb{Z}(\{0,\pm N\})$. Any positive definite sequence $(c_n)_{n \in \mathbb{Z}}$ satisfying $c_0=1$ and having nonzero coefficients only at $\pm N$ yields the trigonometric polynomial
		\begin{equation}\label{eq:pd-cn}
			1+c_N e^{iNt}+c_{-N}e^{-iNt}=1+2\text{Re}(c_N e^{iNt}).
		\end{equation}
		By Lemma~\ref{lem:finite-herglotz}, positive definiteness of the sequence is equivalent to nonnegativity of the polynomial \eqref{eq:pd-cn}, that is
		$$
		1+2\text{Re}(c_N e^{iNt})\ge0, \quad \forall t\in\mathbb{R}.
		$$
		This forces $|c_N|\le 1/2$. Hence
		$$
		\sum_n c_n=1+2\text{Re} c_N\ \le\ 2,
		$$
		The quality holds for $c_{\pm N}=1/2$, so $\mathcal{T}_\mathbb{Z}(\{0,\pm N\})=2$ and $\mathcal{T}(\Omega_N)=2$.
	\end{proof}
	
	
	\begin{example}\label{prop:fivehalves-gain}
	For $\lambda=5/2$, one has
		\begin{equation}\label{eq:fivehalves}
			\mathcal{T}_{\mathbb{R}}(\Omega_{5/2})\ge \frac{559+80\sqrt{42}}{506}>2.
		\end{equation}
	\end{example}
	
	\begin{proof}
		For $\lambda=5/2$, we have $p=5$, $q=2$, and
		$$
		\mathcal{T}_\mathbb{R}(\Omega_{5/2})=\frac{1}{2}\mathcal{T}_\mathbb{Z}(A),
		\qquad
		A=\{0,\pm1,\pm4,\pm5,\pm6\}.
		$$
		Therefore
		$$
		\mathcal{T}_\mathbb{R}(\Omega_{5/2})>2
		\quad\Longleftrightarrow\quad
		\mathcal{T}_\mathbb{Z}(A)>4.
		$$
		We only give an explicit certification for $\mathcal{T}_\mathbb{Z}(A)>4$.	
		By Lemma~\ref{lem:finite-herglotz}, a finitely supported sequence $c$ with $c_0=1$ is admissible for $\mathcal{T}_\mathbb{Z}(A)$ if and only if ts associated trigonometric polynomial satisfies
		$$
		P_c(t)=\sum_{n\in A}c_n e^{int}\ge0, \qquad \forall t\in \mathbb{R}.
		$$
	    For a real symmetric coefficient sequence $c$, we rewrite
		$$
		P_c(t)
		=1+2c_1\cos t+2c_4\cos(4t)+2c_5\cos(5t)+2c_6\cos(6t).
		$$
		Let $x=\cos t$ and denote
		by $
		T_k(\cos t)=\cos(kt).
		$ the $k$-th Chebyshev polynomial of the first kind.
		If
		$$
		P(x)=1+\alpha_1T_1(x)+\alpha_4T_4(x)+\alpha_5T_5(x)+\alpha_6T_6(x)
		$$
		is nonnegative for every $-1\le x\le1$, then the sequence defined by
		$$
		c_0=1,\qquad c_{\pm k}=\frac{\alpha_k}{2}\quad(k=1,4,5,6),
		\qquad c_n=0\quad(n\notin A)
		$$
		is admissible for $\mathcal{T}_\mathbb{Z}(A)$. Its corresponding value is
		$$
		\sum_{n\in A}c_n
		=1+\alpha_1+\alpha_4+\alpha_5+\alpha_6
		=P(1),
		$$
		since $T_k(1)=1$ for every $k$.
		Now define
		$$
		q(x)=x^3+\frac{\sqrt{42}}{12}x^2-\frac{2}{3}x-\frac{7\sqrt{42}}{144},
		\qquad
		P(x)=\frac{3456}{253}q(x)^2.
		$$
		Then $P(x)\ge0$ on $[-1,1]$.	
		A direct expansion gives
		\begin{equation*}\label{eq:fivehalves-power-expansion}
				P(x)
				=\frac{343}{253}
				+\frac{224\sqrt{42}}{253}x
				+\frac{360}{253}x^2
				-\frac{720\sqrt{42}}{253}x^3  
				-\frac{3600}{253}x^4
				+\frac{576\sqrt{42}}{253}x^5
				+\frac{3456}{253}x^6.
		\end{equation*}
		Using
		\[
		\begin{aligned}
			T_1(x)&=x,\\
			T_4(x)&=8x^4-8x^2+1,\\
			T_5(x)&=16x^5-20x^3+5x,\\
			T_6(x)&=32x^6-48x^4+18x^2-1,
		\end{aligned}
		\]
	we can obtain that
		\begin{equation}\label{eq:fivehalves-square}
			P(x)
			=1+\frac{4\sqrt{42}}{23}T_1(x)
			+\frac{18}{23}T_4(x)
			+\frac{36\sqrt{42}}{253}T_5(x)
			+\frac{108}{253}T_6(x).
		\end{equation}
		In particular, the coefficients of $T_2$ and $T_3$ are zero. From \eqref{eq:fivehalves-square}, the corresponding feasible sequence satisfies
		\[
		c_{\pm1}=\frac{2\sqrt{42}}{23},\qquad
		c_{\pm4}=\frac{9}{23},\qquad
		c_{\pm5}=\frac{18\sqrt{42}}{253},\qquad
		c_{\pm6}=\frac{54}{253}.
		\]
		Therefore
		\[
		\mathcal{T}_\mathbb{Z}(A)\ge P(1)
		=1+\frac{4\sqrt{42}}{23}
		+\frac{18}{23}
		+\frac{36\sqrt{42}}{253}
		+\frac{108}{253}
		=\frac{559+80\sqrt{42}}{253}.
		\]
		Consequently
		\[
		\mathcal{T}_\mathbb{R}(\Omega_{5/2})
		=\frac{1}{2}\mathcal{T}_\mathbb{Z}(A)
		\ge \frac{559+80\sqrt{42}}{506}.
		\]
		Finally,
		\[
		\frac{559+80\sqrt{42}}{506}>2
		\]
		is equivalent to $80\sqrt{42}>453$, and this follows after squaring:
		\[
		80^2\cdot42=268800>453^2=205209.
		\]
		This proves the proposition.
	\end{proof}
	
\subsection{The behavior of $\mathcal{T}_\mathbb{R} (\Omega_\lambda)$ with $\lambda$}
From our computations of $\mathcal{T}_{\mathbb{R}}(\Omega_{\lambda})$  in the previous section, we discover some interesting behavior of this quantity as a function of $\lambda$. We first note that $\mathcal{T}_{\mathbb{R}}(\Omega_{\lambda}) \geq 2$ whenever $\lambda>2$. Indeed, for $0<\varepsilon<1$, set
$\Omega_{\lambda,\varepsilon}=(0,1-\varepsilon)\cup(\lambda,\lambda+1-\varepsilon).$ The associated normalized autocorrelation function $\mathbf{1}_{\Omega_{\lambda,\varepsilon}}\ast\mathbf{1}_{\Omega_{\lambda,\varepsilon}}$ is Tur\'an admissible and has integral equal to $2(1-\varepsilon)$. Sending $\varepsilon\to0$ yields $\mathcal{T}_{\mathbb{R}}(\Omega_{\lambda})\ge 2$. This provides a universal lower bound for the Tur\'an constant on unions of three equal length intervals. What is more interesting is that $\mathcal{T}_{\mathbb{R}}(\Omega_\lambda)$ is not continuous at some integers, but rather jumps at these points.




\begin{proposition}\label{prop:jump-at-two}
	One has
	\begin{equation}\label{eq:jump-at-two}
		\lim_{\lambda\to 2^+} \mathcal{T}_{\mathbb{R}} (\Omega_\lambda) = 2 < 3 = \lim_{\lambda\to 2^-} \mathcal{T}_{\mathbb{R}}(\Omega_\lambda).
	\end{equation}
	In particular, $\lambda=2$ is not a continuous point of $\mathcal{T}_{\mathbb{R}} (\Omega_\lambda)$.
\end{proposition}

\begin{proof}
		Fix $a \in (1,\lambda-1)$ arbitrarily, 
	$$
	\left(L\mathbb{Z}+a \right) \cap \Omega_\lambda = \emptyset, \quad \forall L>a+\lambda+1.
	$$
	Let $f$ be a Tur\'an admissible function with respect to $\Omega_\lambda$. Without loss of generality, we assume that $f$ is real and even. Define
	$$
	P(x)=\sum_{k\in \mathbb{Z}}f(x+kL).
	$$
	Then $P(x)$ is positive definite, $P(0)=1$, $P(a)=0$ and
	\begin{equation}\label{eq:periodized-integral-exact}
		\int_0^L P(x)\,dx=\int_\mathbb{R} f(x)\,dx.
	\end{equation}
	
	By Bochner-Herglotz's theorem, we have
	$$
	P(x)=\sum_{n\in\mathbb{Z}}\gamma_n e^{2\pi i n x/L}, \quad x\in [0,L]
	$$
	with $\gamma_n\ge0$ and $\sum_n\gamma_n=P(0)=1$. Moreover
	$$
	\gamma_0=\frac1L\int_0^L P(x)\,dx.
	$$
	Since $P(0)=1$ and $P(a)=0$, 
	$$
	\gamma_0\le\sum_{n\ne0}\gamma_n=1-\gamma_0.
	$$
	It follows that $\gamma_0\le1/2$. Then combining with \eqref{eq:periodized-integral-exact}, we have
	$$
	\int_\mathbb{R} f(x)\,dx=L\gamma_0\le\frac L2.
	$$
	Taking the supremum over $f$ first, and then letting $L \to (a+\lambda+1)^+$ and $a \to 1^+$, so that
	$$
	\mathcal{T}_\mathbb{R}(\Omega_\lambda)\le\frac{\lambda+2}{2}.
	$$
	Letting $\lambda\to2^+$ gives
	$$
	\lim_{\lambda\to2^+} \mathcal{T}_\mathbb{R}(\Omega_\lambda)=2.
	$$
	On the other hand, Proposition~\ref{prop:connected} gives $ \lim_{\lambda\to2^-}\mathcal{T}_\mathbb{R} \left( (-\lambda-1,\lambda+1) \right) = 3$. This proves \eqref{eq:jump-at-two}.
\end{proof}

Furthermore, we provide a more precise upper-bound estimate which implies that the Tur\'an constant on $\Omega_{\lambda}$ tends to $2$ as $\lambda\to\infty$. Let $\Omega \subset \mathbb{R}^d$ be an origin-symmetric open set. In the paper \cite{KR2006}, the authors proved that the Tur\'an constant $\mathcal{T}_{\mathbb{R}}(\Omega)$ has a upper bound $1 / \rho$ where $\rho > 0$ is the upper density of a set $\Lambda$ with $\Omega \cap (\Lambda-\Lambda) \subseteq\{0\}$. 

\begin{proposition}\label{upper-bound}
	One also has
	\begin{equation}
		\lim_{\lambda\to +\infty} \mathcal{T}_{\mathbb{R}} (\Omega_\lambda) = 2.
	\end{equation}

\end{proposition}

\begin{proof}
	For each odd $k \in \mathbb{Z}$ and $\lambda > 5$, one could get $\lfloor \lambda-5 \rfloor +1$ points in the interval 
	$$\left[ \left(\lambda-1\right)k+2, \left(\lambda-1\right) \left(k+1\right) -2\right]$$ 
	such that the distant between every two points is not equal to the value in $(0,1) \cup (\lambda-1, \lambda+1)$. We denote the set of the above points $\Lambda_k$.
	
	Define the discrete set 
	$$
	\Lambda = (\cup_{k \in 2\mathbb{Z}+1} \Lambda_k) \bigcup  (\lambda-1)\mathbb{Z}.
	$$
	Then $\Omega_\lambda \cap(\Lambda-\Lambda) \subseteq \{0\}$. Now we estimate the upper density of $\Lambda$. For any $h>0$, there exist an integer $l$ such that $ h \in [\left(\lambda-1\right)l,  \left(\lambda-1\right) \left(l+1\right))$. Therefore, 
	\begin{align*}
		\sup_{x \in \mathbb{R}} \frac{|\Lambda \cap [-h,h]|}{2h} &\geq \frac{|\Lambda \cap [-h,h] +x |}{2h}\\
		&\geq \frac{|\Lambda \cap [-\left(\lambda-1\right)l,\left(\lambda-1\right)l]|}{2\left(\lambda-1\right)(l+1)}\\
		&= \frac{l(\lfloor \lambda-5 \rfloor +1)+2l+1}{2\left(\lambda-1\right)(l+1)}.
	\end{align*}
	So we have
	\begin{align*}
		\operatorname{dens} (\Lambda) &= \limsup_{h \rightarrow \infty} \sup_{x \in \mathbb{R}} \frac{|\Lambda \cap [-h,h]|}{2h}\\
		&\geq \lim_{k \rightarrow \infty} \frac{l(\lfloor \lambda-5 \rfloor +1)+2l+1}{2\left(\lambda-1\right)(l+1)}\\
		&= \frac{\lfloor \lambda-5 \rfloor +3}{2\lambda-2}.
	\end{align*}
	Therefore
	$$
	\mathcal{T}_\mathbb{R}(\Omega_\lambda)\leq \frac{1}{\operatorname{dens} (\Lambda)} \leq \frac{2\lambda-2}{\lfloor \lambda-5 \rfloor +3}.
	$$
	Letting $\lambda\to+\infty$ gives
	$$
	\lim_{\lambda\to+\infty} \mathcal{T}_\mathbb{R}(\Omega_\lambda)=2.
	$$
\end{proof}
	
	\section{Discussion}
	The results above reduce the continuous problem to finite arithmetic models, leaving two natural directions for further investigation. 
	First, one may study the asymptotic behavior of $\mathcal{T}_\mathbb{R}(\Omega_\lambda)$ as $\lambda\to\infty$ along rational parameters $p/q$ with bounded $q$. The rigidity inequality $\mathcal{T}_\mathbb{R}(\Omega_N)=2$ suggests the possibility of arithmetic oscillation.
	Second, the finite-union principle extends to sets with multiple supports. For the set
	$$
	\Omega_\lambda^{(k)}=[-1,1]\cup\bigcup_{j=1}^k\bigl([\lambda_j-1,\lambda_j+1]\cup[-\lambda_j-1,-\lambda_j+1]\bigr),
	$$
	one obtains a limit of finite discrete Tur\'an problems with $2k+1$ lattice blocks, which become commensurate under a finite union principle. It therefore remains to evaluate these finite arithmetic models.

\appendix
\section{A semidefinite-program formulation of the discrete Tur\'an problem}\label{app:sdp}
	The arguments in the main text use only the positive-definite-sequence formulation of the discrete Tur\'an problem. For completeness, we record here an equivalent finite semidefinite-program (SDP) formulation. This provides a convenient computational representation of the finite problems arising from Theorems~\ref{thm:rational} and~\ref{thm:limit}, but it is not needed in the proofs of the examples in Subsection~\ref{subsec:examples}.
	
	Let $A\subset\mathbb{Z}$ be a finite symmetric set with $0\in A$, and put
	\[
	m=\max\{|n|:n\in A\}.
	\]
	By Lemma~\ref{lem:finite-herglotz},
	\[
	\mathcal{T}_{\mathbb{Z}}(A)
	=
	\sup\left\{
	\sum_{n\in A}c_n:
	P_c(t):=\sum_{n\in A}c_ne^{int}\ge0\ \text{for every }t\in\mathbb{R},\quad c_0=1
	\right\}.
	\]
	Since $e^{in\cdot0}=1$, the quantity being maximized is simply
	\[
	\sum_{n\in A}c_n=P_c(0).
	\]
	Thus the only apparently infinite constraint is the requirement $P_c(t)\ge0$ for every real $t$.
	
	The Fej\'er--Riesz theorem~\cite{Fejer1916,Riesz1916} converts this nonnegativity condition into a finite matrix condition. Define
	\[
	v(t)=(1,e^{it},\ldots,e^{imt})^{\top}.
	\]
	A nonnegative trigonometric polynomial of degree at most $m$ admits a Gram representation
	\begin{equation}\label{eq:appendix-gram}
		P_c(t)=v(t)^*Qv(t)
	\end{equation}
	for some Hermitian positive semidefinite matrix $Q=(Q_{jk})_{0\le j,k\le m}$. Here
	\[
	Q\succeq0
	\]
	means that
	\[
	z^*Qz\ge0\qquad\text{for every }z\in\mathbb{C}^{m+1};
	\]
	it does not mean that the individual entries of $Q$ are nonnegative.
	
	Expanding the right-hand side of \eqref{eq:appendix-gram} gives
	\[
	v(t)^*Qv(t)
	=
	\sum_{j,k=0}^{m}Q_{jk}e^{i(k-j)t}
	=
	\sum_{\ell=-m}^{m}
	\left(
	\sum_{\substack{0\le j,k\le m\\k-j=\ell}}Q_{jk}
	\right)e^{i\ell t}.
	\]
	Therefore the Fourier coefficient at frequency $\ell$ is the sum of the entries of $Q$ along the diagonal $k-j=\ell$. In particular, the normalization $c_0=1$ becomes
	\[
	\operatorname{tr}Q=\sum_{j=0}^{m}Q_{jj}=1,
	\]
	and the requirement that no frequencies outside $A$ occur becomes
	\[
	\sum_{\substack{0\le j,k\le m\\k-j=\ell}}Q_{jk}=0,
	\qquad
	\ell\in\{-m,\ldots,m\}\setminus A.
	\]
	Finally, if $\mathbf{1}=(1,\ldots,1)^{\top}$, then $v(0)=\mathbf{1}$ and hence
	\[
	P_c(0)=\mathbf{1}^*Q\mathbf{1}.
	\]
	Consequently the discrete Tur\'an constant is exactly the optimal value of the finite SDP
	\begin{equation}\label{eq:appendix-sdp}
		\boxed{
			\begin{aligned}
				\mathcal{T}_{\mathbb{Z}}(A)=\max_{Q}\quad
				&\mathbf{1}^*Q\mathbf{1}\\
				\text{subject to}\quad
				&Q\succeq0,\\
				&\operatorname{tr}Q=1,\\
				&\sum_{\substack{0\le j,k\le m\\k-j=\ell}}Q_{jk}=0,
				\qquad \ell\in\{-m,\ldots,m\}\setminus A.
		\end{aligned}}
	\end{equation}
	This is a semidefinite program in the standard sense: the objective is linear in the entries of $Q$, all displayed equality constraints are linear, and the remaining constraint requires $Q$ to lie in the positive-semidefinite cone. Hence every finite discrete problem produced by the rational reduction in Theorem~\ref{thm:rational}, and every finite approximating problem in Theorem~\ref{thm:limit}, has an explicit finite SDP representation.
	
	\bigskip

	\textbf{Acknowledgments}
	Following the reduction of the continuous Tur\'an problem for \(\Omega_{\lambda}\) to finite discrete Tur\'an problems in Theorem \ref{thm:sandwich}, generative AI tools were consulted to investigate computational methods for the relevant finite optimization problems. Notably, they proposed the semidefinite programming formulation adopted for the problems studied in \textbf{Section~4}. 
	
	For the case $\lambda=5/2$ in \textbf{Example~2}, generative AI served as a computational aid to search for an explicit feasible nonnegative polynomial certificate. This included identifying the polynomials 
	\[
	q(x)
	=
	x^3+\frac{\sqrt{42}}{12}x^2-\frac{2}{3}x
	-\frac{7\sqrt{42}}{144},
	\qquad
	P(x)=\frac{3456}{253}q(x)^2,
	\]
	and supporting the associated symbolic calculations such as
	polynomial and Chebyshev expansion, checking cancellation of the forbidden $T_2$- and $T_3$-coefficients, and computing the
	lower bound from $P(1)$.
	
	All mathematical derivations, exact identities, proofs and conclusions presented in the manuscript were independently examined and fully verified by the authors. Generative AI tools were additionally employed for linguistic refinement across the manuscript.


\end{document}